\documentclass[pdftex,11pt]{amsart} 
\usepackage[pdftex]{hyperref} 
\usepackage{amsmath,amssymb}
\usepackage{mathtools}
\usepackage{amscd}
\usepackage{mathrsfs}
\usepackage[foot]{amsaddr}
\usepackage{amsthm}
\usepackage{setspace}
\usepackage[all]{xy}
\usepackage{comment}
\usepackage[pdftex]{graphicx} 
\usepackage{tikz}
\usetikzlibrary{arrows}
\usetikzlibrary{shapes.geometric,fit}
\usepackage{tikz-cd}
\usepackage{fancyhdr}

\theoremstyle{definition}
\newtheorem{thm}{Theorem}[subsection]
\newtheorem*{thm*}{Theorem}
\newtheorem{defi}[thm]{Definition}
\newtheorem*{defi*}{Definition}
\newtheorem*{acknowledge}{Acknowledgement}
\newtheorem{cor}[thm]{Corollary}
\newtheorem*{cor*}{Corollary}
\newtheorem{prop}[thm]{Proposition}
\newtheorem*{prop*}{Proposition}
\newtheorem{lem}[thm]{Lemma}
\newtheorem*{lem*}{Lemma}
\newtheorem{ex}[thm]{Example}
\newtheorem*{ex*}{Example}

\newtheorem*{rem*}{Remark}

\newtheorem*{hw*}{Homework}

\newcommand{\C}{\mathbb{C}}
\newcommand{\Clif}{\mathrm{Clif}}

\newcommand{\fix}{\mathrm{fix}}

\newcommand{\N}{\mathbb{N}}
\newcommand{\T}{\mathbb{T}}

\DeclareMathOperator{\Iso}{\mathrm{Iso}}

\DeclareMathOperator{\FIS}{\mathrm{FIS}}
\DeclareMathOperator{\FCIS}{\mathrm{FCIS}}

\DeclareMathOperator{\ab}{\mathrm{ab}}

\def\i<#1>{\langle #1 \rangle}
\def\l<#1>{\left\langle #1 \right\rangle}

\makeatletter
\renewenvironment{proof}[1][\proofname]{\par
  \normalfont
  \topsep6\p@\@plus6\p@ \trivlist
  \item[\hskip\labelsep{\bfseries #1}\@addpunct{.}]\ignorespaces
}{%
  \endtrivlist
}
\renewcommand{\proofname}{\sc{Proof}}
\makeatother
\makeatletter
\newcommand*{\defeq}{\mathrel{\rlap{%
                     \raisebox{0.3ex}{$\m@th\cdot$}}%
                     \raisebox{-0.3ex}{$\m@th\cdot$}}%
                     =}
\makeatother

\title[]{Invariant sets and normal subgroupoids of universal \'etale groupoids induced by congruences of inverse semigroups}
\author{Fuyuta Komura}
\address{Department of Mathematics, Faculty of Science and Technology, Keio University,
	3-–14-–1 Hiyoshi, Kohoku-ku, Yokohama, 223-–8522, Japan}
\address{ Mathematical Science Team
Center for Advanced Intelligence Project (AIP)
RIKEN}
\address{Phone: +81-45-566-1641+42706}
\address{Fax: +81-45-566-1642}
\email{fuyuta.k@keio.jp}

\begin{document}
\maketitle
\begin{abstract}

For a given inverse semigroup,
one can associate an \'etale groupoid which is called the universal groupoid.
Our motivation is studying the relation between inverse semigroups and associated \'etale groupoids.
In this paper, we focus on congruences of inverse semigroups,
which is a fundamental concept to consider quotients of inverse semigroup.
We prove that a congruence of an inverse semigroup induces a closed invariant set and a normal subgroupoid of the universal groupoid.
Then we show that the universal groupoid associated to a quotient inverse semigroup is described by the restriction and quotient of the original universal groupoid.
Finally we compute invariant sets and normal subgroupoids induced by special congruences including abelianization.

\end{abstract}

\section{Introduction}

The relation among inverse semigroups, \'etale groupoids and C*-algebras has been revealed by many researchers.
Paterson associated the universal groupoid to an inverse semigroup in \cite{paterson2012groupoids}.
Paterson proved that the C*-algebras associated to the universal groupoids are isomorphic to the inverse semigroup C*-algebras.
Furthermore Paterson showed that the universal groupoid has the universal property about ample actions to totally disconnected spaces (also see \cite{STEINBERG2010689}).
In this paper we investigate a relation between inverse semigroups and the universal groupoids.
In particular we focus on congruences of inverse semigroups, which is a fundamental concept to consider quotient inverse semigroups.
Indeed congruences of inverse semigroups induce closed invariant subsets and normal subgroupoids of the universal groupoids.
Our main theorem is that the universal groupoid of a quotient inverse semigroup is isomorphic to the restriction and quotient of the original universal groupoid (Theorem \ref{Main theorem}).

This paper is organized as follows.
In Section 1, we review the notions of inverse semigroups and \'etale groupoids.
In Section 2, we deal with some least congruences,
although some statements may be well-known for specialists.
We also show that there is a one-to-one correspondence between normal congruences of semilattices and some invariant subsets (Corollary \ref{Correspondens between normal cong and invariant subset}).  
In Section 3, 
we prove that the universal groupoid of a quotient inverse semigroup is isomorphic to the restriction and quotient of the original universal groupoid.
Then we deal with some special congruences.
In Section 4, we apply propositions and theorems obtained in Section 2 and Section 3.
We prove that the number of fixed points of a Boolean action is equal to or less than the number of certain semigroup homomorphisms (Corollary \ref{the number of fixed point}).

\begin{acknowledge}
	The author would like to thank Prof.\ Takeshi Katsura for his support and encouragement.
	This work was supported by RIKEN Junior Research Associate Program.
\end{acknowledge}

\section{Preliminaries}

\subsection{Inverse semigroups}

We recall basic facts about inverse semigroups.
See \cite{Lawson} or \cite{paterson2012groupoids} for more details.
An inverse semigroup $S$ is a semigroup where for every $s\in S$ there exists a unique $s^*\in S$ such that $s=ss^*s$ and $s^*=s^*ss^*$.
We denote the set of all idempotents in $S$ by $E(S)\defeq\{e\in S\mid e^2=2\}$.
A zero element is a unique element $0\in S$ such that $0s=s0=0$ holds for all $s\in S$.
A unit is a unique element $1\in S$ such that $1s=s1=s$ holds for all $s\in S$.
An inverse semigroup dose not necessarily have a zero element or a unit. 
It is known that $E(S)$ is a commutative subsemigroup of $S$.
A subsemigroup of $S$ is a subset of $S$ which is closed under the product and inverse of $S$.
A subsemigroup $N$ of $S$ is said to be normal if $E(S)\subset N$ and $sns^*\in N$ hold for all $s\in S$ and $n\in N$.
An order on $E(S)$ is defined by $e\leq f$ if $ef=e$.

An equivalence relation $\nu$ on $S$ is called a congruence if $(s,t)\in \nu$ implies $(as,at),(sa,ta)\in \nu$ for all $s,t,a\in S$.
The quotient set $S/\nu$ becomes an inverse semigroup such that the quotient map $q\colon S\to S/\nu$ is a semigroup homomorphism.
A congruence $\rho$ on $E(S)$ is said to be normal if $(e,f)\in \rho$ implies $(ses^*,sfs^*)\in \rho$ for all $e,f\in E(S)$ and $s\in S$.
One of typical examples of normal congruences is $E(S)\times E(S)$.
Assume that $\rho$ is a normal congruence on $E(S)$.
Define an equivalence relation $\nu_{\rho,\min}$ on $S$ by declaring that $(s,t)\in \nu_{\rho,\min}$ if $(s^*s,t^*t)\in \rho$ and $se=te$ holds for some $e\in E(S)$ with $(e,s^*s)\in \rho$.
Then $\nu_{\rho,\min}$ is the minimum congruence on $S$ such that its restriction to $E(S)$ coincides with $\rho$.
One can see that $\nu_{E(S)\times E(S), \min}$ is the least congruence such that the quotient inverse semigroup becomes a group. 
We call $S/\nu_{E(S)\times E(S), \min}$ the maximal group image of $S$.

An inverse semigroup $S$ is said to be Clifford if $s^*s=ss^*$ holds for all $s\in S$.
A congruence $\nu$ on $S$ is said to be Clifford if $S/\nu$ is Clifford.

\subsection{\'Etale groupoids}

We recall that the notion of \'etale groupoids.
See \cite{renault1980groupoid} and \cite{asims} for more details.

A groupoid is a set $G$ together with a distinguished subset $G^{(0)}\subset G$,
source and range maps $d,r\colon G\to G^{(0)}$ and a multiplication 
\[
G^{(2)}\defeq \{(\alpha,\beta)\in G\times G\mid d(\alpha)=r(\beta)\}\ni (\alpha,\beta)\mapsto \alpha\beta \in G
\]
such that
\begin{enumerate}
	\item for all $x\in G^{(0)}$, $d(x)=x$ and $r(x)=x$ hold,
	\item for all $\alpha\in G$, $\alpha d(\alpha)=r(\alpha)\alpha=\alpha$ holds,
	\item for all $(\alpha,\beta)\in G^{(2)}$, $d(\alpha\beta)=d(\beta)$ and $r(\alpha\beta)=r(\alpha)$ hold,
	\item if $(\alpha,\beta),(\beta,\gamma)\in G^{(2)}$,
	we have $(\alpha\beta)\gamma=\alpha(\beta\gamma)$,
	\item\label{inverse} every $\gamma \in G$,
	there exists $\gamma'\in G$ which satisfies $(\gamma',\gamma), (\gamma,\gamma')\in G^{(2)}$ and $d(\gamma)=\gamma'\gamma$ and $r(\gamma)=\gamma\gamma'$.   
\end{enumerate}
Since the element $\gamma'$ in (\ref{inverse}) is uniquely determined by $\gamma$,
$\gamma'$ is called the inverse of $\gamma$ and denoted by $\gamma^{-1}$.
We call $G^{(0)}$ the unit space of $G$.
A subgroupoid of $G$ is a subset of $G$ which is closed under the inversion and multiplication. 
For $U\subset G^{(0)}$, we define $G_U\defeq d^{-1}(U)$ and $G^{U}\defeq r^{-1}(U)$.
We define also $G_x\defeq G_{\{x\}}$ and $G^x\defeq G^{\{x\}}$ for $x\in G^{(0)}$.
The isotropy bundle of $G$ is denoted by $\Iso(G)\defeq\{\gamma\in G\mid d(\gamma)=r(\gamma)\}$.
A subset $F\subset G^{(0)}$ is said to be invariant if $d(\alpha)\in F$ implies $r(\alpha)\in F$ for all $\alpha\in G$.
We say that $x\in G^{(0)}$ is a fixed point if $\{x\}\subset G^{(0)}$ is invariant.
If $G$ satisfies $G=\Iso(G)$,
$G$ is called a group bundle over $G^{(0)}$.
A group bundle $G$ is said to be abelian if $G_x$ is an abelian group for all $x\in G^{(0)}$.

A topological groupoid is a groupoid equipped with a topology where the multiplication and the inverse are continuous.
A topological groupoid is said to be \'etale if the source map is a local homeomorphism.
Note that the range map of an \'etale groupoid is also a local homeomorphism.
The next proposition easily follows from the definition of \'etale groupoids.

\begin{prop}\label{condition that groupoid hom is conti}
	Let $G$ and $H$ be \'etale groupoids.
	A groupoid homomorphism $\Phi\colon G\to H$ is a locally homeomorphism if and only if $\Phi|_{G^{(0)}}\colon G^{(0)}\to H^{(0)}$ is a locally homeomorphism.
\end{prop}

\subsubsection{Quotient \'etale groupoids}
We recall that the notion of quotient \'etale groupoids.
See \cite[Section 2]{KOMURA} for more details.
Let $G$ be an \'etale groupoid.
We say that a subgroupoid $H\subset G$ is normal if
\begin{itemize}
\item $G^{(0)}\subset H \subset \Iso(G)$ holds, and
\item for all $\alpha\in G$, $\alpha H\alpha^{-1}\subset H$ holds.
\end{itemize}
For a normal subgroupoid $H\subset G$,
define an equivalence relation $\sim$ on $G$ by declaring that $\alpha\sim \beta$ if $d(\alpha)=d(\beta)$ and $\alpha\beta^{-1}\in H$ hold.
Then $G/H\defeq G/{\sim}$ becomes a groupoid such that the quotient map is a groupoid homomorphism.
If $H$ is open in $G$,
then $G/H$ is an \'etale groupoid with respect to the quotient topology.
Moreover the quotient map is a local homeomorphism  (see \cite[Subsection 2.1]{KOMURA} for these facts).

We have the fundamental homomorphism theorem.
The proof is left to the reader.
\begin{prop}\label{homomorphism theorem}
Let $G$ and $H$ be \'etale groupoids and $\Phi\colon G\to H$ be a continuous groupoid homomorphism which is a local homeomorphism.
Assume that $\Phi$ is injective on $G^{(0)}$.
Then $\ker \Phi\defeq \Phi^{-1}(H^{(0)})$ is an open normal subgroupoid of $G$.
Moreover there exist an isomorphism $\widetilde{\Phi}\colon G/\ker\Phi\to \Phi(G)$ which makes the following diagram commutative;
\begin{center}
	\begin{tikzpicture}[auto]
	\node (a) at (0,0) {$G$};
	\node (c) at (3,0){$H$};
	\node (d) at (0,-2) {$G/\ker\Phi$};
	\draw[->] (a) to node {$\Phi$} (c) ;
	\draw[->,swap] (a) to node {$Q$} (d);
	\draw[->,swap] (d) to node {$\widetilde{\Phi}$} (c);
	\end{tikzpicture}
	,
\end{center}
where $Q\colon G\to G/\ker\Phi$ denotes the quotient map.	
\end{prop}

For an \'etale groupoid $G$,
the author of \cite{KOMURA} constructed the \'etale abelian group bundle $G^{\ab}$.
We briefly recall the definition of $G^{\ab}$.

We define $G_{\fix}\defeq G_F$, where $F\subset G^{(0)}$ denotes the set of all fixed points.
Then $G_{\fix}$ is a group bundle.
Define $[G_{\fix},G_{\fix}]\defeq \bigcup_{x\in F}[(G_{\fix})_x,((G_{\fix}))_x]$,
where $[(G_{\fix})_x,(G_{\fix})_x]$ denotes the commutator subgroup of $(G_{\fix})_x$.
Then $[G_{\fix},G_{\fix}]$ is an open normal subgroupoid of $G_{\fix}$ (\cite[Proposition 3.2.1]{KOMURA}).
Now we define $G^{\ab}\defeq G_{\fix}/[G_{\fix},G_{\fix}]$.

The reason why we consider $G^{\ab}$ is that the next theorem holds.
We denote the universal groupoid C*-algebra of $G$ by $C^*(G)$.
See \cite{paterson2012groupoids} for the definition of the universal groupoid C*-algebras.
\begin{thm}[{\cite[Theorem 3.2.4.]{KOMURA}}]
	Let $G$ be an \'etale groupoid such that $G^{(0)}$ is a locally compact Hausdorff space with respect to the relative topology of $G$.
	Then $C^*(G^{\ab})$ is the abelianization of $C^*(G)$ as a C*-algebra.
\end{thm}
\subsection{\'Etale groupoids associated to inverse semigroup actions}

Let $X$ be a topological space.
We denote by $I_X$ the inverse semigroup of homeomorphisms between open sets in $X$.
An action $\alpha\colon S\curvearrowright X$ is a semigroup homomorphism $S\ni s\mapsto \alpha_s\in I_X$.
For $e\in E(S)$, we denote the domain of $\beta_e$ by $D_e^{\alpha}$.
Then $\alpha_s$ is a homeomorphism from $D_{s^*s}^{\alpha}$ to $D_{ss^*}^{\alpha}$.
In this paper, we always assume that $\bigcup_{e\in E(S)}D_e^{\alpha}=X$ holds. 

For an action $\alpha\colon S\curvearrowright X$,
we associate an \'etale groupoid $S\ltimes_{\beta}X$ as the following.
First we put the set $S*X\defeq \{(s,x) \in S\times X \mid x\in D^{\alpha}_{s^*s}\}$.
Then we define an equivalence relation $\sim$ on $S*X$ by $(s,x)\sim (t,y)$ if
\[
\text{$x=y$ and there exists $e\in E(S)$ such that $x\in D^{\alpha}_e$ and $se=te$}.  
\]
Set $S\ltimes_{\alpha}X\defeq S*X/{\sim}$ and denote the equivalence class of $(s,x)\in S*X$ by $[s,x]$.
The unit space $S\ltimes_{\alpha}X$ is $X$, where $X$ is identified with a subset of $S\ltimes_{\alpha}X$ via the injection
\[
 X\ni x\mapsto [e,x] \in S\ltimes_{\alpha}X, x\in D^{\alpha}_e.
\]
The source map and range maps are defined by
\[
d([s,x])=x, r([s,x])=\alpha_s(x)
\]
for $[s,x]\in S\ltimes_{\alpha}X$.
The product of $[s,\alpha_t(x)],[t,x]\in S\ltimes_{\alpha}X$ is $[st,x]$.
The inverse should be $[s,x]^{-1}=[s^*,\alpha_s(x)]$.
Then $S\ltimes_{\alpha}X$ is a groupoid in these operations.
For $s\in S$ and an open set $U\subset D_{s^*s}^{\alpha}$,
define 
\[[s, U]\defeq \{[s,x]\in S\ltimes_{\alpha}X\mid x\in U\}.\]
These sets form an open basis of $S\ltimes_{\alpha}X$.
In these structures,
$S\ltimes_{\alpha}X$ is an \'etale groupoid.

Let $S$ be an inverse semigroup.
Now we define the spectral action $\beta\colon S\curvearrowright \widehat{E(S)}$.
A character on $E(S)$ is a nonzero semigroup homomorphism from $E(S)$ to $\{0,1\}$,
where $\{0,1\}$ is an inverse semigroup in the usual product.
The set of all characters on $E(S)$ is denoted by $\widehat{E(S)}$.
We see $\widehat{E(S)}$ as a locally compact Hausdorff space with respect to the topology of pointwise convergence.
Define
\[N^{e}_P\defeq\{\xi\in\widehat{E(S)}\mid \xi(e)=1, \xi(p)=0 \text{ for all }p\in P\}\]
for $e\in E(S)$ and a finite subset $P\subset E(S)$,
then these sets form an open basis of $\widehat{E(S)}$. 
For $e\in E(S)$,
we define $D_e^{\beta}\defeq\{\xi\in\widehat{E(S)}\mid \xi(e)=1\}$.
For each $s\in S$ and $\xi\in D_{s^*s}^{\beta}$, define $\beta_s(\xi)\in D_{ss^*}^{\beta}$ by $\beta_s(\xi)(e)=\xi(s^*es)$, where $e\in E(S)$.
Then $\beta$ is an action $\beta\colon S\curvearrowright \widehat{E(S)}$, which we call the spectral action of $S$.
Now the universal groupoid of $S$ is defined to be $G_u(S)\defeq S\ltimes_{\beta}\widehat{E(S)}$.


\section{Certain least congruences}

It is known that every inverse semigroup admits the least Clifford congruence and the least commutative congruence.
For example, see \cite[Proposition III.\ 6.\ 7]{petrich1984inverse} for the least Clifford congruence and \cite{Piochi1986} for the least commutative congruence.
In this section,
we  that every inverse semigroup admits the least Clifford congruence and the least commutative congruence,
based on the spectrum.

\subsection{Invariant subset of $\widehat{E(S)}$}

Let $S$ be an inverse semigroup.
A subset $F\subset \widehat{E(S)}$ is said to be invariant if $\beta_s(F\cap D_{s^*s})\subset F$ holds for all $s\in S$.
Note that $F$ is invariant if and only if $F$ is invariant as a subset of the groupoid $G_u(S)$.
We omit the proof of the next proposition.
\begin{prop}
	Let $S$ be an inverse semigroup and $F\subset \widehat{E(S)}$ be an invariant subset.
	We define the set $\rho_F\subset E(S)\times E(S)$ of all pairs $(e,f)\in E(S)\times E(S)$ such that $\xi(e)=\xi(f)$ holds for all $\xi\in F$.
	Then $\rho_F$ is a normal congruence on $E(S)$.
\end{prop}

Let $S$ be an inverse semigroup and $\rho$ be a normal congruence on $E(S)$.
Moreover let $q\colon E(S)\to E(S)/{\rho}$ denote the quotient map.
For $\xi\in E(S)/{\rho}$,
we define $\widehat{q}(\xi)\in \widehat{E(S)}$ by $\widehat{q}(\xi)(e)=\xi(q(e))$, where $e\in E(S)$. 
Since $q$ is surjective, $\widehat{q}(\xi)$ is not zero.
Then $\widehat{q}\colon \widehat{E(S)/{\rho}}\to \widehat{E(S)}$ is a proper continuous map.
Therefore $F_{\rho}\defeq \widehat{q}(\widehat{E(S)/{\rho}})$ is a closed subset.

We say that $F\subset \widehat{E(S)}$ is multiplicative if the multiplication of two elements in $F$ also belongs to $F$ whenever it is not zero.
 
\begin{prop}\label{F_rho is a multiplicative closed}
	Let $S$ be an inverse semigroup and $\rho$ be a normal congruence on $E(S)$.
	Then $F_{\rho}\subset \widehat{E(S)}$ is a closed multiplicative invariant set.
\end{prop}

\begin{proof}
	It is easy to show that $F_{\rho}\subset\widehat{E(S)}$ is a closed multiplicative set.
	We show that $F_{\rho}\subset \widehat{E(S)}$ is invariant.
	Take $\xi\in F_{\rho}$ and $s\in S$ with $\xi(s^*s)=1$.
	By the definition of $F_{\rho}$,
	there exists $\eta\in\widehat{E(S)/\rho}$ such that $\xi=\eta\circ q$, where $q\colon E(S)\to E(S)/\rho$ denotes the quotient map.
	Then we have 
	\[\beta_s(\xi)(e)=\xi(s^*es)=\eta(q(s^*es))=\beta_{q(s)}(\eta)(q(e))=\widehat{q}(\beta_{q(s)}(\eta))(e)\]
	for all $e\in E(S)$.
	Therefore it follows that $\beta_s(\xi)=\widehat{q}(\beta_{q(s)}(\eta))\in F_{\rho}$.
\qed
\end{proof}

\begin{prop}\label{nu=nu_F_nu}
Let $S$ be an inverse semigroup.
Then $\rho=\rho_{F_{\rho}}$ holds for every normal congruence $\rho$ on $E(S)$.
\end{prop}
\begin{proof}
	Assume that $(e,f)\in\rho$.
	For all $\eta\in\widehat{E(S)/\rho}$,
	we have 
	\[\widehat{q}(\eta)(e)=\eta(q(e))=\eta(q(f))=\widehat{q}(\eta)(f).\]
	Therefore $(e,f)\in \rho_{F_{\rho}}$.
	
	To show the reverse inclusion,
	assume that $(e,f)\in \rho_{F_{\rho}}$.
	Define $\eta_{q(e)}\in \widehat{E(S)/{\rho}}$ by
	\begin{align*}
	\eta_{q(e)}(p)=
	\begin{cases}
	1 & (p\geq q(e)), \\
	0 & (\text{otherwise}),
	\end{cases}
	\end{align*}
	where $p\in E(S)/\rho$.
	By $(e,f)\in \rho_{F_{\rho}}$,
	we have $\eta_{q(e)}(q(f))=\eta_{q(e)}(q(e))=1$.
	Therefore $q(f)\geq q(e)$.
	Similarly we obtain $q(f)\leq q(e)$,
	so $q(e)=q(f)$ holds.
	Thus it follows that $(e,f)\in \rho$.
	\qed
\end{proof}

We say that $F\subset \widehat{E(S)}$ is unital if $F$ contains the constant function $1$.

\begin{lem}\label{Stone-Wierstrass type lemma}
	Let $S$ be an inverse semigroup and $F\subset \widehat{E(S)}$ be a unital multiplicative set.
	Assume that $F$ separates $E(S)$ (i.e.\ for $e,f\in E(S)$, $e=f$ is equivalent to the condition that  $\xi(e)=\xi(f)$ holds for all $\xi\in F$).
	Then $F$ is dense in $\widehat{E(S)}$. 
\end{lem}

\begin{proof}
	For $e\in E(S)$ and a finite subset $P\subset E(S)$,
	we define
	\[N^{e}_P\defeq\{\xi\in\widehat{E(S)}\mid \xi(e)=1, \xi(p)=0 \text{ for all }p\in P\}.\]
	Recall that these sets form an open basis of $\widehat{E(S)}$.
	Observe that $N^{e}_P=N^{e}_{eP}$ holds, where $eP\defeq\{ep\in E(S)\mid p\in P\}$.
	Now	it suffices to show that $F\cap N^{e}_P\not=\emptyset$ holds for nonempty $N^{e}_P$
	such that $p\leq e$ holds for all $p\in P$.
	
	In case that $P=\emptyset$ the constant function $1$ belongs to $F\cap N^{e}_P$.
	We may assume that $p\leq e$ holds for all $p\in P$.
	Since $N^{e}_P$ is nonempty,
	we have $e\not=p$ for all $p\in P$.
	Since $F$ separates $E(S)$,
	there exists $\xi_p\in F$ such that $\xi_p(e)=1$ and $\xi_p(p)=0$ for each $p\in P$.
	Define $\xi\defeq \prod_{p\in P} \xi_p$, then $\xi\in N^{e}_P\cap F$.
\qed	
\end{proof}
\begin{prop}\label{F=F_nu_F}
	Let $S$ be an inverse semigroup.
	Then $F=F_{\rho_F}$ holds for every unital multiplicative invariant closed set $F\subset \widehat{E(S)}$.
	
\end{prop}

\begin{proof}
	It is easy to show that $F\subset F_{\rho_F}$.
	Let $q\colon E(S)\to E(S)/{\rho_F}$ denote the quotient map.
	Then the set $\widehat{q}^{-1}(F)$ is a unital multiplicative closed set which separates $E(S)/{\rho_F}$.
	By Lemma \ref{Stone-Wierstrass type lemma},
	$\widehat{q}^{-1}(F)=\widehat{E(S)/{\rho_F}}$ holds.
	Therefore, we have $F\supset\widehat{q}(\widehat{q}^{-1}(F))=\widehat{q}(\widehat{E(S)/{\rho_F}})=F_{\rho_F}$.
	\qed
\end{proof}

\begin{cor}\label{Correspondens between normal cong and invariant subset}
	Let $S$ be an inverse semigroup.
	There is a one-to-one correspondence between normal congruences on $E(S)$ and unital multiplicative invariant closed set in $\widehat{E(S)}$. 
\end{cor}
\begin{proof}
	Just combine Proposition \ref{nu=nu_F_nu} and Proposition \ref{F=F_nu_F}.
	\qed
\end{proof}


\subsection{The least Clifford congruences}

Let $S$ be an inverse semigroup.
Recall that a congruence $\rho$ on $S$ is said to be Clifford if $S/\rho$ is Clifford.
For example,
$S\times S$ is a Clifford congruence on $S$.
In this subsection, we prove that every inverse semigroup admits the least Clifford congruence (Theorem \ref{Cliffordization of inverse semigroup}).
Our construction of the congruence is based on the fixed point of $\widehat{E(S)}$.

\begin{defi}
	Let $S$ be an inverse semigroup.
	A character $\xi\in \widehat{E(S)}$ is said to be fixed if $\xi(s^*es)=\xi(e)$ holds for all $e\in E(S)$ and $s\in S$ such that $\xi(s^*s)=1$.
	We denote the set of all fixed characters by $\widehat{E(S)}_{\fix}$.
\end{defi}

One can see that $\widehat{E(S)}_{\fix}$ is a closed subset of $\widehat{E(S)}$.
Moreover, $\widehat{E(S)}_{\fix}$ is a multiplicative set.
A fixed character is characterized as the next proposition.

\begin{prop}\label{extension of fixed character}
	Let $S$ be an inverse semigroup.
	Then,
	$\xi\in \widehat{E(S)}$ is fixed if and only if $\xi$ is extended to a semigroup homomorphism $\widetilde{\xi}\colon S\to \{0,1\}$.
	In this case,
	$\widetilde{\xi}\colon S\to \{0,1\}$ is a unique extension of $\xi$.
\end{prop}

\begin{proof}
	If $\xi\in \widehat{E(S)}$ has an extension $\widetilde{\xi}\colon S\to\{0,1\}$,
	we have \[\widetilde{\xi}(s)=\widetilde{\xi}(s)^2=\widetilde{\xi}(s^*s)=\xi(s^*s)\] for all $s\in S$.
	Therefore,
	a semigroup homomorphism extension of $\xi$ is unique.
	
	It is obvious that $\xi$ is fixed if $\xi$ has a semigroup homomorphism extension.
	Assume that $\xi\in \widehat{E(S)}$ is fixed.
	Then,
	define $\widetilde{\xi}(s)\colon S\to\{0,1\}$ by $\widetilde{\xi}(s)\defeq \xi(s^*s)$ for $s\in S$.
	For $s,t\in S$,
	if $\xi(t^*t)=1$,
	we have $\widetilde{\xi}(st)=\xi(t^*s^*st)=\xi(s^*s)= \widetilde{\xi}(s) \widetilde{\xi}(t)$.
	If $\xi(t^*t)=0$,
	we have $\widetilde{\xi}(st)=\widetilde{\xi}(s)\widetilde{\xi}(t)=0$.
	Thus,
	$\widetilde{\xi}$ is a semigroup homomorphism.
	\qed
\end{proof}

\begin{defi}\label{Cliffordization of inverse semigroup}
	Let $S$ be an inverse semigroup.
	We define a normal congruence $\rho_{\Clif}\defeq \rho_{\widehat{E(S)}_{\fix}}$ on $E(S)$.
	Furthermore, we define a congruence $\nu_{\Clif}\defeq \nu_{\rho_{\Clif},\min}$ on $S$ and $S^{\Clif}\defeq S/{\nu_{\Clif}}$.
\end{defi}


\begin{lem}\label{pull back of the spec of clif is fixed}
	Let $S$ be an inverse semigroup , $\nu$ be a Clifford congruence on $S$ and $q\colon S\to S/\nu$ be the quotient map.
	Then a set \[F_{\nu}= \{\xi\circ q\in \widehat{E(S)}\mid \xi\in \widehat{E(S/\nu)}\}\] is contained in $\widehat{E(S)}_{\fix}$. 
	Moreover,
	$\widehat{E(S)}_{\fix}=\widehat{E(S)}$ holds if and only if $S$ is Clifford.
\end{lem}
\begin{proof}
	Take $\xi\in \widehat{E(S/\nu)}$ and assume that $\xi(q(s^*s))=1$ for some $s\in S$.
	For all $e\in E(S)$,
	we have 
	\[q(s^*es)=q(es)^*q(es)=q(es)q(es)^*=q(ess^*e)=q(ss^*)q(e)=q(s^*s)q(e)\] since $S/\nu$ is Clifford.
	Now we have
	\[\xi\circ q(s^*es)=\xi(q(s^*es))=\xi(q(s^*s)q(e))=\xi(q(s^*s))\xi(q(e))=\xi\circ q(e).\]
	Therefore $\xi\circ q$ is a fixed character.
	
	Applying what we have shown for the trivial congruence $\nu=\{(s,s)\in S\times S\mid s\in S\}$,
	it follows that $\widehat{E(S)}_{\fix}=\widehat{E(S)}$ holds if $S$ is Clifford.
	Assume that $\widehat{E(S)}_{\fix}=\widehat{E(S)}$ holds and take $s\in S$.
	Define a character $\xi_{s^*s}\in\widehat{E(S)}$ by
	\begin{align*}
	\xi_{s^*s}(e)=
	\begin{cases}
	1 & (e\geq s^*s), \\
	0 & (\text{otherwise}),
	\end{cases}
	\end{align*}
	where $e\in E(S)$.
	Since we assume that $\widehat{E(S)}_{\fix}=\widehat{E(S)}$ and $\xi_{s^*s}=1$,
	we have \[\xi_{s^*s}(ss^*)=\xi_{s^*s}(s^*(ss^*)s)=\xi_{s^*s}(s^*s)=1.\]
	Then we have $s^*s\leq ss^*$.
	It follows that $s^*s\geq ss^*$ from the same argument.
	Now we have $s^*s=ss^*$ and $S$ is Clifford.
	\qed
\end{proof}

Now, we show that every inverse semigroup admits the Cliffordization.

\begin{thm}
	Let $S$ be an inverse semigroup.
	Then $\nu_{\Clif}$ is the least Clifford congruence on $S$.
\end{thm}

\begin{proof}
	First, we show a congruence $\nu_{\Clif}$ is Clifford.
	Take $s\in S$ and $\xi\in \widehat{E(S)}_{\fix}$.
	Then one can see $\xi(s^*s)=\xi(ss^*)$.
	Therefore $(s^*s,ss^*)\in\nu_{\Clif}$ and $\nu_{\Clif}$ is a Clifford congruence.
	
	Let $\nu$ be a Clifford congruence and $q\colon S\to S/\nu$ be the quotient map.
	To show $\nu_{\Clif}\subset \nu$,
	take $(s,t)\in \nu_{\Clif}$.
	First,
	we show that $(s^*s,t^*t)\in \nu$.
	We define $\eta\in\widehat{E(S/{\nu})}$ by
	\begin{align*}
	\eta(e)=
	\begin{cases}
	1 & (e\geq q(s^*s) ), \\
	0 & (\text{otherwise}).
	\end{cases}
	\end{align*}
	By Lemma \ref{pull back of the spec of clif is fixed}, it follows that $\eta\circ q\in \widehat{E(S)}_{\fix}$.
	Since $(s,t)\in \nu_{\Clif}$,
	we have $1=\eta \circ q(s^*s)=\eta\circ q(t^*t)$,
	which implies $q (t^*t)\geq q(s^*s)$.
	The reverse inequality is obtained symmetrically,
	and therefore $q(t^*t)= q(s^*s)$ holds.
	
	Let $\eta\in \widehat{E(S/{\nu_{\Clif}})}$ be the above character.	
	Since $\eta\circ q$ is a fixed character and $(s,t)\in \nu_{\Clif}$,
	there exists $e\in E(S)$ such that $\eta\circ q(e)=1$ and $se=te$ hold.
	Since $\eta\circ q(e)=1$,
	we have $q(e)\geq q(s^*s)=q(t^*t)$ by the definition of $\eta$.
	Now we have $q(s)=q(s)q(e)=q(t)q(e)=q(t)$.
	Therefore $(s,t)\in\nu$.
	\qed 
\end{proof}

\begin{cor}
	Let $S$ be an inverse semigroup, $T$ be a Clifford inverse semigroup and $\varphi\colon S\to T$ be a semigroup homomorphism.
	Then there exists a unique semigroup homomorphism $\widetilde{\varphi}\colon S^{\Clif}\to T$ which makes the following diagram commutative;
	\begin{center}
		\begin{tikzpicture}[auto]
		\node (a) at (0,0) {$S$};
		\node (c) at (3,0){$T$};
		\node (d) at (0,-2) {$S^{\Clif}$};
		\draw[->] (a) to node {$\varphi$} (c) ;
		\draw[->,swap] (a) to node {$q$} (d);
		\draw[->,swap] (d) to node {$\widetilde{\varphi}$} (c);
		\end{tikzpicture}
		,
	\end{center}
where $q\colon S\to S^{\Clif}$ denotes the quotient map.
\end{cor}

\subsection{The least commutative congruences}

We say that a congruence on inverse semigroup is commutative if the quotient semigroup is commutative.
In this subsection,
we show that every inverse semigroups admits the least commutative congruence.

We denote the circle group by $\T\defeq\{z\in\C\mid \lvert z\rvert=1\}$.
We see $\T\cup \{0\}$ as an inverse semigroup in the usual products.
By $\widehat{S}$,
we denote the set of all semigroup homomorphisms from $S$ to $\T\cup\{0\}$.

\begin{defi}\label{def of nu_ab}
	Let $S$ be an inverse semigroup.
	We define a commutative congruence $\nu_{\ab}$ on $S$ to be the set of all pairs $(s,t)\in S\times S$ such that $\varphi(s)=\varphi(t)$ holds for all $\varphi\in \widehat{S}$.
	We define $S^{\ab}\defeq S/{\nu_{\ab}}$.
\end{defi}

Let $S$ be a Clifford inverse semigroup and $e\in E(S)$.
We define $H_e\defeq \{s\in S\mid  s^*s=e\}$.
One can see that $H_e$ is a group in the operation inherited from $S$.
Note that the unit of $H_e$ is $e$.

In order to show that $\nu_{\ab}$ is the least commutative congruence,
we need the next lemma.

\begin{lem}\label{ext of character}
	Let $S$ be a Clifford inverse semigroup and $e\in E(S)$.
	Then,
	a group homomorphism $\varphi\colon H_e\to \T$ is extended to a semigroup homomorphism $\widetilde{\varphi}\colon S\to \T\cup\{0\}$.
\end{lem}

\begin{proof}
	Define 
	\begin{align*}
	\widetilde{\varphi}(s)=
	\begin{cases}
	\varphi(se) & (s^*s\geq e), \\
	0 & (\text{otherwise}).
	\end{cases}
	\end{align*}
	Then, one can check that $\widetilde{\varphi}$ is a semigroup homomorphism extension of $\varphi$.
	\qed 
\end{proof}

\begin{thm}\label{nu_ab is the least commutative congruence}
	Let $S$ be an inverse semigroup.
	Then,
	$\nu_{\ab}$ is the least commutative congruence on $S$.
\end{thm}
\begin{proof}
	Assume that $\nu$ is a commutative congruence.
	Let $q\colon S\to S/\nu$ denote the quotient map.
	In order to show $\nu_{\ab}\subset \nu$,
	take $(s,t)\in\nu_{\ab}$.
	
	First, we show that $q(s^*s)=q(t^*t)$.
	It suffices to show that $\xi(q(s^*s))=\xi(q(t^*t))$ holds for all $\xi\in \widehat{E(S/\nu)}$.
	Note that $\xi\circ q\in \widehat{E(S)}$ is a fixed point by Lemma \ref{pull back of the spec of clif is fixed}.
	Since $\xi\circ q$ is a restriction of an element in $\widehat{S}$ by Proposition \ref{extension of fixed character},
	$\xi(q(s^*s))=\xi(q(t^*t))$ follows from $(s^*s,t^*t)\in\nu_{\ab}$.
	
	In order to show that $q(s)=q(t)$,
	it suffices to show that $\psi(q(s))=\psi(q(t))$ for all group homomorphisms $\psi\colon H_{q(s^*s)}\to \T$,
	since $H_{q(s^*s)}=\{a\in S/\nu \mid a^*a=q(s^*s)\}$ is an abelian group.
	By Lemma \ref{ext of character},
	there exists a semigroup homomorphism extension $\widetilde{\psi}\in \widehat{S/\nu}$ of $\psi$.
	Since $\widetilde{\psi}\circ q \in \widehat{S}$ and $(s,t)\in\nu_{\ab}$,
	we have $\psi(q(s))=\psi(q(t))$.
	Therefore,
	$q(s)=q(t)$ holds.
	\qed
\end{proof}

\begin{cor}
		Let $S$ be an inverse semigroup, $T$ be a commutative inverse semigroup and $\varphi\colon S\to T$ be a semigroup homomorphism.
	Then there exists a unique semigroup homomorphism $\widetilde{\varphi}\colon S^{\Clif}\to T$ which makes the following diagram commutative;
	\begin{center}
		\begin{tikzpicture}[auto]
		\node (a) at (0,0) {$S$};
		\node (c) at (3,0){$T$};
		\node (d) at (0,-2) {$S^{\ab}$};
		\draw[->] (a) to node {$\varphi$} (c) ;
		\draw[->,swap] (a) to node {$q$} (d);
		\draw[->,swap] (d) to node {$\widetilde{\varphi}$} (c);
		\end{tikzpicture}
		,
	\end{center}
	where $q\colon S\to S^{\ab}$ denotes the quotient map.
\end{cor}
\section{Universal \'etale groupoids associated to quotient inverse semigroups}

\subsection{General case}

Let $S$ be an inverse semigroup and $\nu$ be a congruence on $S$.
Let $q\colon S\to S/\nu$ denote the quotient map. 
Note that
\[F_{\nu}=\{\xi\circ q\in E(S)\mid \xi\in \widehat{E(S)}\}\]
is a closed invariant subset of $G_u(S)$ as shown in Proposition \ref{F_rho is a multiplicative closed}.

We omit the proof of the next proposition.
\begin{prop}
	Let $S$ be an inverse semigroup and $H\subset S$ be a subsemigroup such that $E(S)\subset H$.
	Then the map
	\[
	G_u(H)\ni [s,\xi]\mapsto [s,\xi]\in G_u(S)
	\]
	is a groupoid homomorphism which is a homeomorphism onto its image.
	Moreover the image is an open subgroupoid of $G_u(S)$.
\end{prop}
Via the map in the above proposition,
we identify $G_u(H)$ with an open subgroupoid of $G_u(S)$.

Let $S$ be an inverse semigroup, $\nu$ be a congruence on $S$ and $q\colon S\to S/\nu$ be the quotient map.
Define $\ker\nu\defeq q^{-1}(E(S/\nu))\subset S$.
Then $\ker\nu$ is a normal subsemigroup of $S$.
Although $G_u(\ker\nu)$ is not necessarily a normal subgroupoid of $G_u(S)$,
the following holds.

\begin{prop}
	Let $S$ be an inverse semigroup and $\nu$ be a congruence on $S$.
	Then $G_u(\ker\nu)_{F_{\nu}}$ is an open normal subgroupoid of $G_u(S)_{F_{\nu}}$.
\end{prop}

\begin{proof}
	Now we know that $G_u(\ker\nu)_{F_{\nu}}$ is an open normal subgroupoid of $G_u(S)_{F_{\nu}}$. 
	We show that $G_u(\ker\nu)_{F_{\nu}}$ is normal in $G_u(S)_{F_{\nu}}$.
	Let $q\colon S\to S/\nu$ denote the quotient map.
	
	First we show $G_u(\ker\nu)_{F_{\nu}}\subset \Iso(G_u(S)_{F_{\nu}})$.
	Take $[n,\xi]\in G_u(\ker\nu)_{F_{\nu}}$, where $n\in\ker\nu$.
	Since $\xi\in F_{\nu}$ holds,
	there exists $\eta\in \widehat{E(S/\nu)}$ such that $\xi=\eta\circ q$.
	Since $q(n)\in E(S/\nu)$ holds,
	we have $q(n^*)\in E(S/\nu)$ and
	\begin{align*}
	\beta_n(\xi)(e)&=\xi(n^*en)=\eta(q(n^*)q(e)q(n)) \\
	&=\eta(q(n^*))\eta(q(e))\eta(q(n))=\eta(q(n^*n))\eta(q(e))=\xi(e)
	\end{align*}
	for all $e\in E(S)$.
	Therefore $\beta_n(\xi)=\xi$ holds and it follows that $[n,\xi]\in \Iso(G_u(\ker\nu)_{F_{\nu}})$.
	
	Next we show that $[s,\eta][n,\xi][s,\eta]^{-1}\in G_u(\ker\nu)_{F_{\nu}}$ holds for all $[n,\xi]\in G_u(\ker\nu)_{F_{\nu}}$ and $[s,\eta]\in G_u(S)_{F_{\nu}}$ such that $\eta=\beta_n(\xi)$.
	One can see that
	\[[s,\eta][n,\xi][s,\eta]^{-1}=[sns^*,\beta_s(\eta)].\]
	Now it follows that $[s,\eta][n,\xi][s,\eta]^{-1}\in G_u(\ker\nu)_{F_{\nu}}$ from $sns^*\in\ker\nu$.
	\qed  
\end{proof}

\begin{thm}\label{Main theorem}
	Let $S$ be an inverse semigroup and $\nu$ be a congruence on $S$.
	Then $G_u(S/\nu)$ is isomorphic to $G_u(S)_{F_{\nu}}/G_u(\ker{\nu})_{F_{\nu}}$.
\end{thm}
\begin{proof}
	Let $q\colon S\to S/\nu$ denote the quotient map.
	Note that a map 
	\[\widehat{q}\colon \widehat{E(S/\nu)}\ni \xi \mapsto \xi\circ q\in F_{\nu}\]
	is a homeomorphism.
	Define a map 
	\[\Phi\colon G_u(S)_{F_{\nu}}\ni [s,\widehat{q}(\xi)]\mapsto [q(s),\xi]\in G_u(S/\nu).\]
	Using Proposition \ref{condition that groupoid hom is conti}, one can see that $\Phi$ is a groupoid homomorphism which is a local homeomorphism and injective on $G_u(S)_{F_{\nu}}^{(0)}$.
	Observe that $\Phi$ is surjective.
	
	
	We show that $\ker\Phi=G_u(\ker\nu)_{F_{\nu}}$ holds.
	The inclusion $\ker\Phi\supset G_u(\ker\nu)_{F_{\nu}}$ is obvious.
	In order to show $\ker\Phi\subset G_u(\ker\nu)_{F_{\nu}}$, take $[s,\widehat{q}(\xi)]\in \ker\Phi$.
	Since we have $[q(s),\xi]\in G_u(S/\nu)^{(0)}$ and $q(E(S))=E(S/\nu)$,
	there exists $e\in E(S)$ such that $[q(s),\xi]=[q(e),\xi]$.
	There exists $f\in E(S)$ such that $\xi(q(f))=1$ and $q(s)q(f)=q(e)q(f)$.
	Now we have $sf\in\ker\nu$,
	so it follows that \[ [s,\widehat{q}(\xi)]=[sf,\widehat{q}(\xi)] \in G_u(\ker\nu)_{F_{\nu}}. \]
	This shows that $\ker\Phi=G_u(\ker\nu)_{F_{\nu}}$.

	By Proposition \ref{homomorphism theorem},
	$\Phi$ induces an isomorphism $\widetilde{\Phi}$ which makes the following diagram commutative
	\begin{center}
		\begin{tikzpicture}[auto]
		\node (a) at (0,0) {$G_u(S)_{F_{\nu}}$};
		\node (b) at (3,0) {$G_u(S/\nu)$};
		\node (c) at (0,-2) {$G_u(S)_{F_{\nu}}/G_u(\ker{\nu})_{F_{\nu}}$};
		\draw[->] (a) to node {$\Phi$}  (b);
		\draw[->,swap] (a) to node {$Q$} (c);
		\draw[->,swap] (c) to node {$\widetilde{\Phi}$} (b);
		\end{tikzpicture}
		,
	\end{center}
	where $Q$ denotes the quotient map.
	\qed
\end{proof}

\subsection{Universal groupoids associated to special quotient inverse semigroups}

\subsubsection{Minimum congruences associated to normal congruences on semilattices of idempotents}

Let $S$ be an inverse semigroup.
Recall that a congruence $\rho$ on $E(S)$ is normal if $(e,f)\in \rho$ implies $(ses^*, sfs^*)\in \rho$ for all $s\in S$ and $e,f\in E(S)$.
Note that one can construct the least congruence $\nu_{\rho,\min}$ whose restriction to $E(S)$ coincides with $\rho$. 
Recall that we can associate the closed invariant subset $F_{\rho}$ of $G_u(S)$ as shown in Proposition \ref{F_rho is a multiplicative closed}.

\begin{prop}\label{restriction by min congruence}
	Let $S$ be an inverse semigroup and $\rho$ be a normal congruence on $E(S)$.
	Then $G_u(S/\nu_{\rho,\min})$ is isomorphic to $G_u(S)_{F_{\rho}}$. 
\end{prop}

\begin{proof}
	By Theorem \ref{Main theorem}, it suffices to show that $G_u(\ker\nu_{\rho,\min})_{F_{\rho}}=G_u(S)_{F_{\rho}}^{(0)}$ holds.
	Let $q\colon S\to S/\nu_{\rho,\min}$ denote the quotient map.
	Take $[n,\widehat{q}(\xi)]\in G_u(\ker\nu_{\rho,\min})_{F_{\rho}}$, where $n\in \ker\nu_{\rho,\min}$ and $\xi\in \widehat{E(S/\rho)}$.
	Since $n\in \ker\nu_{\rho,\min}$, there exists $e\in E(S)$ such that $q(n)=q(e)$.
	By the definition of $\nu_{\rho,\min}$,
	there exists $f\in E(S)$ such that $nf=ef$ and $(n^*n,f)\in \rho$ hold.
	Observe that $\widehat{q}(\xi)(n^*n)=\xi(q(n^*n))=\xi(q(f))=\xi(q(e))=1$.
	We have
	\[
	[n,\widehat{q}(\xi)]=[nf,\widehat{q}(\xi)]=[ef,\widehat{q}(\xi)]\in G_u(S)_{F_{\rho}}^{(0)}.
	\] 
	Now we have shown that $G_u(\ker\nu_{\rho,\min})_{F_{\rho}}=G_u(\ker\nu_{\rho,\min})_{F_{\rho}}^{(0)}$.\qed
\end{proof}

\begin{thm}
	Let $S$ be an inverse semigroup.
	Then $G_u(S^{\Clif})$ is isomorphic to $G_u(S)_{\fix}$.
\end{thm}
\begin{proof}
	Recall the definition of $\nu_{\Clif}=\nu_{\rho_{\Clif},\min}$ (see Definition \ref{Cliffordization of inverse semigroup}).
	Since we have Proposition \ref{restriction by min congruence},
	it suffices to show $F_{\rho_{\Clif}}=\widehat{E(S)}_{\fix}$.
	By Lemma \ref{pull back of the spec of clif is fixed},
	we have $F_{\rho_{\Clif}}\subset\widehat{E(S)}_{\fix}$.
	To show the reverse inclusion,
	take $\xi\in \widehat{E(S)}_{\fix}$.
	By Proposition \ref{extension of fixed character},
	there exists a semigroup homomorphism extension $\widetilde{\xi}\colon S\to \{0,1\}$.
	Since $\{0,1\}$ is Clifford,
	there exists a semigroup homomorphism $\eta\colon S^{\Clif}\to \{0,1\}$ such that $\eta\circ q=\widetilde{\xi}$, where $q\colon S\to S^{\Clif}$ denotes the quotient map.
	Therefore we have $\xi=\eta\circ q|_{E(S)}\in F_{\rho_{\Clif}}$.
	Now we have shown $F_{\rho_{\Clif}}=\widehat{E(S)}_{\fix}$.
	\qed
\end{proof}

\subsubsection{The least commutative congruences}

Let $S$ be an inverse semigroup and $\nu_{\ab}$ be the least commutative congruence (see Proposition \ref{def of nu_ab} and Theorem \ref{nu_ab is the least commutative congruence}).
Recall that the abelianization of $S$ is defined to be $S^{\ab}\defeq S/\nu_{\ab}$.

\begin{thm}
	Let $S$ be an inverse semigroup.
	Then $G_u(S^{\ab})$ is isomorphic to $G_u(S)^{\ab}$.
\end{thm}

\begin{proof}
	By Theorem \ref{Main theorem},
	it suffices to show that $F_{\nu_{\ab}}=\widehat{E(S)}_{\fix}$ and $G_u(\ker\nu_{\ab})_{\fix}=[G_u(S)_{\fix}, G_u(S)_{\fix}]$ hold.
	
	Observe that $\nu_{\ab}$ is equal to $\nu_{\Clif}$ on $E(S)$.
	Indeed this follows from the fact that $\varphi|_{E(S)}\in \widehat{E(S)}_{\fix}$ holds for all $\varphi\in \widehat{S}$.
	Therefore we have $F_{\nu_{\ab}}=\widehat{E(S)}_{\fix}$.
	
	Next we show $G_u(\ker\nu_{\ab})_{\fix}=[G_u(S)_{\fix},G_u(S)_{\fix}]$.
	The inclusion \[G_u(\ker\nu_{\ab})_{\fix}\supset[G_u(S)_{\fix},G_u(S)_{\fix}]\] is easy to show.
	
	Let $q\colon S\to S^{\ab}$ and $q'\colon S\to S^{\Clif}$ denote the quotient maps.
	Since a commutative inverse semigroup is Clifford,
	there exists a semigroup homomorphism $\sigma\colon S^{\Clif}\to S^{\ab}$ such that $q=\sigma\circ q'$.
	To show the reverse inclusion
	\[G_u(\ker\nu_{\ab})_{\fix}\subset[G_u(S)_{\fix},G_u(S)_{\fix}],\]
	take $[n,\widehat{q}(\xi)]\in G_u(\ker\nu_{\ab})_{\fix}$, where $n\in\ker\nu_{\ab}$ and $\xi\in \widehat{E(S^{\ab})}$.
	Since $n\in \ker\nu_{\ab}$,
	there exists $e\in E(S)$ such that $q(n)=q(e)$.
	Then we have $q(n^*n)=q(e)$.
	Since $\nu_{\ab}$ coincides with $\nu_{\Clif}$ on $E(S)$,
	it follows that $q'(n^*n)=q'(e)$.
	Define
	\[
	H_{q'(e)}=\{s\in S^{\Clif}\mid s^*s=q'(e)\},
	\]
	then $H_{q'(e)}$ is a group in the operation inherited from $S^{\Clif}$.
	Observe that a unit of $H_{q'(e)}$ is $q'(e)$ and we have $q'(n)\in H_{q'(e)}$.
	Fix a group homomorphism $\chi\colon H_{q'(e)}\to \T$ arbitrarily.
	By Proposition \ref{ext of character},
	$\chi$ is extended to the semigroup homomorphism $\widetilde{\chi}\colon S^{\Clif}\to \T\cup \{0\}$.
	Since $\T\cup \{0\}$ is commutative,
	there exists a semigroup homomorphism $\overline{\chi}\colon S^{\ab}\to\T\cup\{0\}$ which makes the following diagram commutative;
	\begin{center}
		\begin{tikzpicture}[auto]
		\node (a) at (0,0) {$S$};
		\node (b) at (3,0) {$S^{\Clif}$};
		\node (c) at (3,-2) {$\T\cup\{0\}$};
		\node (d) at (0,-2) {$S^{\ab}$};
		\draw[->] (a) to node {$q'$}  (b);
		\draw[->] (b) to node {$\widetilde{\chi}$} (c);
		\draw[->,swap] (a) to node {$q$} (d);
		\draw[->,swap] (d) to node {$\overline{\chi}$} (c);
		\end{tikzpicture}
		.
	\end{center}
Now we have \[\chi(q'(n))=\overline{\chi}(q(n))=\overline{\chi}(q(e))=\chi(q'(e)).\]
Since we take a group homomorphism $\chi\colon H_{q'(e)}\to \T$ arbitrarily,
it follows that $q'(n)\in [H_{q'(e)},H_{q'(e)}]$, where $[H_{q'(e)},H_{q'(e)}]$ denotes the commutator subgroup of $H_{q'(e)}$.
Therefore there exists $s_1,s_2,\dots,s_m,t_1,t_2,\dots,t_m\in S$ such that
\begin{align*}
q'(n)&=q'(s_1)q'(t_1)q'(s_1)^*q'(t_1)^*\cdots q'(s_m)q'(t_m)'q'(s_m)^*q'(t_m)^* \\
&=q'(s_1t_1s_1^*t_1^*\cdots s_mt_ms_m^*t_m^*).
\end{align*}
By the definition of $\nu_{\Clif}$,
there exists $f\in E(S)$ such that \[nf= s_1t_1s_1^*t_1^*\cdots s_mt_ms_m^*t_m^*f\] and $q'(n^*n)=q'(f)$ hold.
Then we have
\begin{align*}
[n,\widehat{q}(\xi)]&=[nf,\widehat{q}(\xi)] \\
&=[s_1t_1s_1^*t_1^*\cdots s_mt_ms_m^*t_m^*f,\widehat{q}(\xi)] \\
&=[s_1t_1s_1^*t_1^*\cdots s_mt_ms_m^*t_m^*,\widehat{q}(\xi)]\in [G_u(S)_{\fix},G_u(S)_{\fix}]
\end{align*}
Thus it is shown that $G_u(\ker\nu_{\ab})_{\fix}=[G_u(S)_{\fix}, G_u(S)_{\fix}]$.
\qed
\end{proof}

\section{Applications and examples}

\subsection{Clifford inverse semigroups from the view point of fixed points}
A 0-group is an inverse semigroup isomorphic to $\Gamma\cup\{0\}$ for some group $\Gamma$.
It is clear that every 0-group is Clifford inverse semigroup. 
Conversely, we see that every Clifford inverse semigroup is embedded into a direct product of 0-groups.
Remark that this fact is already known (see \cite[Theorem 2.6]{petrich1984inverse}).
Using fixed characters, we obtain a new proof. 

Let $S$ be a Clifford inverse semigroup and $\xi\in \widehat{E(S)}$.
Since $\{\xi\}\subset \widehat{E(S)}$ is invariant by Lemma \ref{pull back of the spec of clif is fixed},
we may consider a normal congruence $\rho_{\xi}\defeq\rho_{\{\xi\}}$ on $E(S)$ and a congruence $\nu_{\xi}\defeq\nu_{\rho_{\{\xi\}},\min}$ on $S$.
If $\xi=1$,
$\rho_{\xi}$ coincides with $E(S)\times E(S)$ and $S/\nu_{\xi}$ is the maximal group image of $S$.
If $\xi\not=1$,
$E(S/\nu_{\xi})$ is isomorphic to $\{0,1\}$ and we identify them.
Thus $(S/\nu_{\xi})\setminus \{0\}$ becomes a group.
Now we define a group $S(\xi)$ by
\begin{align*}
S(\xi)\defeq
\begin{cases}
S/\nu_{\xi} & (\xi=1), \\
(S/\nu_{\xi})\setminus\{0\} & (\xi\not=1).
\end{cases}
\end{align*}
For any $\xi\in\widehat{E(S)}$ with $\xi\not=1$,
note that $S/\nu_{\xi}=S(\xi)\cup\{0\}$ holds and $S/\nu_{\xi}$ is a 0-group.

Let $q_{\xi}\colon S\to S/\nu_{\xi}$ denote the quotient map.
\begin{prop}
Let $S$ be a Clifford inverse semigroup.
A semigroup homomorphism
\[
 \Phi\colon S\ni s\mapsto (q_{\xi}(s))_{\xi\in \widehat{E(S)}}\in \prod_{\xi\in\widehat{E(S)}} S/\nu_{\xi} 
\]
is injective.
In particular,
every Clifford inverse semigroup is embedded into a direct product of 0-groups and groups.
\end{prop}

\begin{proof}
Assume that $s,t\in S$ satisfy $\Phi(s)=\Phi(t)$.
Since $q_{\xi}(s)=q_{\xi}(t)$ holds for all $\xi\in S$,
there exists $e_{\xi}\in E(S)$ such that $\xi(e_{\xi})=1$ and $se_{\xi}=te_{\xi}$.
Therefore we have $(s,t)\in\nu_{\Clif}$.
Since we assume that $S$ is Clifford,
$s=t$ holds.
Thus $\Phi$ is injective.
\qed
\end{proof}

\begin{prop}\label{spectrum of a finitely generated Clifford inverse semigroup is finite}
	Let $S$ be a finitely generated Clifford inverse semigroup.
	Then $\widehat{E(S)}$ is a finite set.
\end{prop}
\begin{proof}
	Take a finite set $F\subset S$ which generates $S$.
	Let $X$ denote the set of all nonzero semigroup homomorphisms from $S$ to $\{0,1\}$.
	Then a map
	\[X \ni \xi \mapsto (\xi(f))_{f\in F}\in\{0,1\}^F\]
	is injective since $F$ generates $S$.
	By Proposition \ref{extension of fixed character} and Lemma \ref{pull back of the spec of clif is fixed},
	a map $X\ni\xi\mapsto \xi_{E(S)}\in \widehat{E(S)}$ is bijective.
	Since $\widehat{E(S)}$ is embedded into $\{0,1\}^F$,
	$\widehat{E(S)}$ is a finite set.
	\qed
\end{proof}

\begin{cor}
	Let $S$ be a finitely generated Clifford inverse semigroup.
	Then $S$ is embedded into a direct sum of finitely many 0-groups and groups.
\end{cor}

Let $S$ be an inverse semigroup and $\xi\in\widehat{E(S)}$.
Recall that $G_u(S)_{\xi}$ is a discrete group.
In \cite{LaLonde2017},
the authors give a way to calculate $G_u(S)_{\xi}$.
Then $\xi^{-1}(\{1\})$ is a directed set with respect to the order inherited from $E(S)$.
For $e,f\in E(S)$ with $e\leq f$,
define a map $\varphi^f_e\colon S(f)\to S(e)$ by $\varphi^f_e(s)=se$ for $s\in S(f)$.
Then $\varphi^{f}_e$ is a group homomorphism.
One can see that $(S(e),\varphi^{f}_e)$ consists an inductive system of groups.
The authors proved the following.
\begin{thm}[{\cite[Theorem 3.1]{LaLonde2017}}] \label{Lalonde}
	Let $S$ be a Clifford inverse semigroup and $\xi\in\widehat{E(S)}$.
	Then there exists an isomorphism
	\[G_u(S)_{\xi} \simeq\varinjlim_{\xi(e)=1}S(e).\]	
\end{thm}

We give a way to realize $\varinjlim_{\xi(e)=1}S(e)$ as a quotient of $S$.

\begin{prop} \label{calculation of ind lim}
	Let $S$ be a Clifford inverse semigroup and $\xi\in\widehat{E(S)}$.
	Then we have the following isomorphism; 
	\[\varinjlim_{\xi(e)=1}S(e)\simeq S(\xi).\]
	
\end{prop}

\begin{proof}
	Let $q\colon S\to S(\xi)$ denote the quotient map and put $\Gamma\defeq \varinjlim S(e)$.
	For $e\in E(S)$ with $\xi(e)=1$,
	we define $\sigma_e\colon S(e)\to S(\xi)$ by $\sigma_e(s)\defeq q(s)$.
	We obtain a group homomorphism $\widetilde{\sigma}\colon \Gamma \to S(\xi)$.
	One can see that $\widetilde{\sigma}$ is an isomorphism.
	\qed
\end{proof}

Combining Theorem \ref{Lalonde} with Proposition \ref{calculation of ind lim},
we obtain the next Corollary.
\begin{cor}\label{calculation of isotropy group}
	Let $S$ be a Clifford inverse semigroup and $\xi\in\widehat{E(S)}$.
	Then $G_u(S)_{\xi}$ is isomorphic to $S(\xi)$.
\end{cor}

Let $I$ be a discrete set and $\{\Gamma_i\}_{i\in I}$ be a family of discrete group.
Then the disjoint union $\coprod_{i\in I}\Gamma_i$ is a discrete group bundle over $I$ in the natural way.
Using Proposition \ref{spectrum of a finitely generated Clifford inverse semigroup is finite} and Corollary \ref{calculation of isotropy group},
we obtain the next corollary.

\begin{cor}\label{universal groupoid associated to finitely generated Clifford inverse semigroup}
	Let $S$ be a finitely generated Clifford inverse semigroup.
	Then there exists an isomorphism
	\[
	G_u(S)\simeq \coprod_{\xi\in\widehat{E(S)}} S(\xi).
	\]
\end{cor}

For an \'etale groupoid $G$ with the locally compact Hausdorff unit space $G^{(0)}$,
we write $C^*(G)$ (resp.\ $C^*_{\lambda}(G)$) to represent the universal (resp.\ reduced) groupoid C*-algebra of $G$ (see \cite{paterson2012groupoids} for the definitions).
Corollary \ref{universal groupoid associated to finitely generated Clifford inverse semigroup} immediately implies the next corollary.

\begin{cor}
	Let $S$ be a finitely generated Clifford inverse semigroup.
	Then we have the two isomorphisms
	\[
	C^*(G_u(S))\simeq \bigoplus_{\xi\in\widehat{E(S)}} C^*(S(\xi)),  C^*_{\lambda}(G_u(S))\simeq \bigoplus_{\xi\in\widehat{E(S)}} C^*_{\lambda}(S(\xi)).
	\]

\end{cor}

\subsection{Free Clifford inverse semigroups}
We investigate universal groupoids and C*-algebras associated to free Clifford inverse semigroups on finite sets.

Firstly, we recall the definition of the free groups.
\begin{defi}
	Let $X$ be a set.
	A free group on $X$ is a pair $(\mathbb{F}(X), \kappa)$ of a group $\FIS(X)$ and a map $\iota\colon X\to \mathbb{F}(X)$ such that
	\begin{enumerate}
		\item $\kappa(X)$ generates $\mathbb{F}(X)$, and
		\item for every group $\Gamma$ and a map $\varphi\colon X\to \Gamma$, 
		there exists a semigroup homomorphism $\widetilde{\varphi}\colon \mathbb{F}(X)\to \Gamma$ such that $\widetilde{\varphi}(x)=\varphi(\kappa(x))$ holds for all $x\in X$.
	\end{enumerate}
\end{defi}

We define free inverse semigroups in the similar way.

\begin{defi}
	Let $X$ be a set.
	A free inverse semigroup on $X$ is a pair $(\FIS(X), \iota)$ of an inverse semigroup $\FIS(X)$ and a map $\iota\colon X\to \FIS(X)$ such that
	\begin{enumerate}
	\item $\iota(X)$ generates $\FIS(X)$, and
	\item for every inverse semigroup $T$ and map $\varphi\colon X\to T$, 
	there exists a semigroup homomorphism $\widetilde{\varphi}\colon \FIS(X)\to T$ such that $\widetilde{\varphi}(x)=\varphi(\iota(x))$ holds for all $x\in X$.
\end{enumerate}
\end{defi}

It is known that a free inverse semigroup uniquely exists up to the canonical semigroup homomorphism.
See \cite[Section 6.1]{Lawson} for the existence of a free inverse semigroup.
The uniqueness is obvious.

\begin{defi}
A free Clifford inverse semigroup on $X$ is a pair $(\FCIS(X),\iota)$ of a Clifford inverse semigroup $\FCIS(X)$ and a map $\iota\colon X\to\FCIS(X)$ such that
\begin{enumerate}
	\item $\iota(X)$ generates $\FCIS(X)$, and
	\item for every Clifford inverse semigroup $T$ and map $\varphi\colon X\to T$, 
	there exists a semigroup homomorphism $\widetilde{\varphi}\colon \FCIS(X)\to T$ such that $\widetilde{\varphi}(x)=\varphi(\iota(x))$ holds for all $x\in X$.
\end{enumerate}
\end{defi}

A free Clifford inverse semigroup uniquely exists up to the canonical isomorphism.
Indeed, for a free Clifford inverse semigroup $(\FIS(X),\iota)$ and the quotient map $q\colon \FIS(X)\to \FIS(X)^{\Clif}$,
one can see that $(\FIS(X)^{\Clif}, q\circ\iota)$ is a free Clifford inverse semigroup on $X$.
The uniqueness is obvious.

Let $X$ be a set.
For $A\subset X$,
define a map $\chi_A\colon X\to \{0,1\}$ by
\begin{align*}
	\chi_A(x)=
	\begin{cases}
		1 & (x\in A), \\
		0 & (x\not\in A).
	\end{cases}
\end{align*}
Since $\{0,1\}$ is Clifford,
$\chi_A$ is extended to the semigroup homomorphism from $\FCIS(X)$ to $\{0,1\}$,
which we also denote by $\chi_A$.
Every semigroup homomorphism from $\FCIS(X)$ to $\{0,1\}$ is given by the form $\chi_A$ with unique $A\subset X$.

By Proposition \ref{extension of fixed character},
$\chi_A|_{E(\FCIS(X))}$ is a fixed character if $A$ is not empty.
By Lemma \ref{pull back of the spec of clif is fixed},
all characters on $E(\FCIS(X))$ are fixed characters.
Therefore we obtain the next proposition.

\begin{prop}\label{spectrum of FCIS}
Let $X$ be a finite set.
Put $S=\FCIS(X)$.
Then a map 
\[ P(X)\setminus \{\emptyset\}\ni A\mapsto \chi_A|_{E(S)}\in \widehat{E(S)} \]
is bijective.
\end{prop}

We identify $\chi_A|_{E(\FCIS(X))}$ with $\chi_A$,
since we can recover $\chi_A$ from the restriction $\chi_A|_{E(\FCIS(X))}$.

For a nonempty set $A\subset X$,
define $e_A\defeq \prod_{x\in A} \iota(x)^*\iota(x)\in E(\FCIS(X))$.
For $e\in E(\FCIS(X))$,
the condition that $\chi_A(e)=1$ is equivalent to the condition that $e\geq e_A$.
Using this fact,
one can prove the next proposition.  
\begin{prop}
A map \[ P(X)\setminus\{\emptyset\} \ni A\mapsto e_A\in E(\FCIS(X))\] is bijective.
\end{prop}

In order to apply Corollary \ref{calculation of isotropy group} for free Clifford inverse semigroups,
we prepare the next proposition.

\begin{prop}\label{S(chiA) is the free group}
	Let $X$ be a set and $A\subset X$ be a nonempty set.
	Put $S=\FCIS(X)$.
	Then $S(\chi_A)$ is isomorphic to the free group $\mathbb{F}(A)$ generated by $A$.
\end{prop}

\begin{proof}
	If $X=A$,
	$S(\chi_A)$ is the maximal group image of $S$.
	Therefore $S(\chi_A)$ is isomorphic to $\mathbb{F}(A)$.
	
	We assume $A\subsetneq X$.
	Let $Q\colon S\to S(\chi_A)$ denote the quotient map.
	By the universality of $\mathbb{F}(A)$,
	define a group homomorphism $\tau\colon \mathbb{F}(A)\to S(\chi_A)$ such that
	$\tau(\kappa(a))=Q(\iota(a))$ for all $a\in A$.
	We construct the inverse map of $\tau$.
	Using the universality of $S=\FCIS(X)$,
	define a semigroup homomorphism $\sigma\colon S \to \mathbb{F}(A)\cup\{0\}$ which satisfies
	\begin{align*}
		\sigma(\iota(x))=
		\begin{cases}
			\kappa(x) & (x\in A) \\
			0 & (x\not\in A)
		\end{cases}
	\end{align*}
	for $x\in X$.
	We claim that $(s,t)\in \nu_{\chi_A}$ implies $\sigma(s)=\sigma(t)$ for $s,t\in S$.
	By $(s,t)\in \nu_{\chi_A}$, we have $se_A=te_A$.
	Since $\sigma(e_A)$ is the unit of $\mathbb{F}(A)$,
	we have $\sigma(s)=\sigma(t)$.
	Therefore we obtain a semigroup homomorphism $\widetilde{\sigma}\colon S(\chi_A)\to \mathbb{F}(A)\cup\{0\}$ which makes the following diagram commutative;
		\begin{center}
		\begin{tikzpicture}[auto]
		\node (a) at (0,0) {$S$};
		\node (c) at (3,0){$\mathbb{F}(A)\cup \{0\}$};
		\node (d) at (0,-2) {$S(\chi_A)\cup\{0\}$};
		\draw[->] (a) to node {$\sigma$} (c) ;
		\draw[->,swap] (a) to node {$Q$} (d);
		\draw[->,swap] (d) to node {$\widetilde{\sigma}$} (c);
		\end{tikzpicture}
		.
	\end{center}
	Now one can verify that $\widetilde{\sigma}|_{S(\chi_A)}$ is the inverse map of $\tau$.
	\qed 
\end{proof}

Now we have the following Theorem.

\begin{thm}
	Let $X$ be a finite set.
	Then there exists an isomorphism
	\[ G_u(\FCIS(X)) \simeq \coprod_{A\in P(X)\setminus\{\emptyset\}} \mathbb{F}(A).\]
\end{thm}

\begin{proof}
	Put $S=\FCIS(X)$.
	By Proposition \ref{spectrum of FCIS},
	it follows that
	\[\widehat{E(S)}=\{\chi_A\in\widehat{E(S)} \mid A\in P(X)\setminus\{\emptyset\}\} \] 
	is a finite set.
	Therefore we have an isomorphism
	\[G_u(S)\simeq \coprod_{A\in P(X)\setminus\{\emptyset\}}G_u(S)_{\chi_A}.\]
	By Proposition \ref{S(chiA) is the free group},
	we obtain the isomorphism in the statement.
	\qed
\end{proof}


\subsection{Fixed points of Boolean actions}

From \cite[Section 5]{STEINBERG2010689},
we recall the notion of Boolean actions.
By a locally compact Boolean space,
we mean a locally compact Hausdorff space which admits a basis of compact open sets.
Let $S$ be an inverse semigroup and $X$ be a locally compact Boolean space.
An action $\alpha\colon S\curvearrowright X$ is said to be Boolean if
\begin{enumerate}
	\item for all $e\in E(S)$, $D^{\alpha}_{e}\subset X$ is a compact open set of $X$, and
	\item a family
	\[\bigg\{D_e^{\alpha}\cup \bigcup_{f\in P}(X\setminus D_f^{\alpha})\mid e\in E(S), \text{$P\subset E(S)$ is a finite set.}\bigg\}\] forms a basis of $X$.  
\end{enumerate}

It is known that $G_u(S)$ has the universality about Boolean actions as follows.

\begin{thm}[{\cite[Proposition 5.5]{STEINBERG2010689}}] \label{universality of Gu(S)}
	Let $S$ be an inverse semigroup, $X$ be a Boolean space and $\alpha\colon S\curvearrowright X$ be a Boolean action.
	Then $S\ltimes_{\alpha} X$ is isomorphic to $G_u(S)_F$ for some closed invariant subset $F\subset \widehat{E(S)}$. 
\end{thm}

\begin{cor} \label{the number of fixed point}
	Let $S$ be a finitely generated inverse semigroup and $\alpha\colon S\curvearrowright X$ be a Boolean action.
	Then $\alpha$ has finitely many fixed points.
	More precisely,
	the number of fixed points of $\alpha$ is equal to or less than the number of nonzero semigroup homomorphisms from $S$ to $\{0,1\}$.   
\end{cor}

\begin{proof}
	Since we assume that $S$ is finitely generated,
	the set of all nonzero semigroup homomorphisms from $S$ to $\{0,1\}$ is a finite set.
	By Proposition \ref{extension of fixed character},
	there exists a bijection between the set of all nonzero semigroup homomorphisms from $S$ to $\{0,1\}$ and $\widehat{E(S)}_{\fix}$.
	Now Theorem \ref{universality of Gu(S)} completes the proof.
	\qed
\end{proof}

\begin{ex}[cf.\ {\cite[Example 3 in Section 4.2]{paterson2012groupoids}}]
	For a natural number $n\in\N$,
	the Cuntz inverse semigroup $S_n$ is an inverse semigroup which is generated by $\{0,1, s_1,\dots,s_n\}$
	with the relation
	\[
	s_i^*s_j= \delta_{i,j}1
	\]
	for all $i,j\in \{1,2,\dots,n\}$.
	Define $\xi\colon S_n\to \{0,1\}$ by $\xi(x)=1$ for all $x\in S_n$.
	Then $\xi$ is the unique nonzero semigroup homomorphism from $S_n$ to $\{0,1\}$.
	Since $0\in S_n$,
	$\xi$ is an isolated point of $\widehat{E(S)}$.
	Therefore,
	every Boolean action of $S_n$ has at most one fixed point,
	which becomes an isolated point.
\end{ex}

\bibliographystyle{plain}
\bibliography{bunken20180828}

\begin{thebibliography}{1}

\bibitem{KOMURA}
F.~Komura.
\newblock Quotients of \'etale groupoids and the abelianizations of groupoid
  {C}*-algebras.
\newblock arXiv:1812.07194, to appear in \textit{Journal of the Australian
  Mathematical Society}.

\bibitem{LaLonde2017}
S.~M. LaLonde and D.~Milan.
\newblock Amenability and uniqueness for groupoids associated with inverse
  semigroups.
\newblock {\em Semigroup Forum}, 95(2):321--344, Oct 2017.

\bibitem{Lawson}
Mark~V Lawson.
\newblock {\em Inverse Semigroups}.
\newblock WORLD SCIENTIFIC, 1998.

\bibitem{paterson2012groupoids}
A.~Paterson.
\newblock {\em Groupoids, Inverse Semigroups, and their Operator Algebras}.
\newblock Progress in Mathematics. Birkh{\"a}user Boston, 2012.

\bibitem{petrich1984inverse}
M.~Petrich.
\newblock {\em Inverse semigroups}.
\newblock Pure and applied mathematics. Wiley, 1984.

\bibitem{Piochi1986}
B.~Piochi.
\newblock Solvability in inverse semigroups.
\newblock {\em Semigroup Forum}, 34(1):287--303, Dec 1986.

\bibitem{renault1980groupoid}
J.~Renault.
\newblock {\em A Groupoid Approach to C*-Algebras}.
\newblock Lecture Notes in Mathematics. Springer-Verlag, 1980.

\bibitem{asims}
A.~Sims.
\newblock Hausdorff \'{e}tale groupoids and their {C}*-algebras.
\newblock {arXiv:1710.10897v1}, 2017.

\bibitem{STEINBERG2010689}
B.~Steinberg.
\newblock A groupoid approach to discrete inverse semigroup algebras.
\newblock {\em Advances in Mathematics}, 223(2):689 -- 727, 2010.

\end{thebibliography}

\end{document}